\documentclass[sn-mathphys]{sn-jnl}

\usepackage{amsmath,amssymb,amsthm}
\usepackage{enumitem}
\usepackage{lmodern}

\newtheorem{theorem}{Theorem}[section]
\newtheorem{lemma}[theorem]{Lemma}
\newtheorem{corollary}[theorem]{Corollary}
\newtheorem{proposition}[theorem]{Proposition}
\newtheorem{conjecture}[theorem]{Conjecture}
\theoremstyle{definition}
\newtheorem{definition}[theorem]{Definition}
\newtheorem{remark}[theorem]{Remark}
\newtheorem{example}[theorem]{Example}

\newcommand{\RR}{\mathbb{R}}
\newcommand{\Rn}{\RR^{n}}
\newcommand{\Rnn}{\RR^{n \times n}}
\newcommand{\bn}{\bar n}
\newcommand{\Pmat}{\mathbf{P}}
\newcommand{\Pzero}{\mathbf{P}_0}
\newcommand{\Nzero}{\mathbf{N}_0}
\newcommand{\Qmat}{\mathbf{Q}}
\newcommand{\Qzero}{\mathbf{Q}_0}
\newcommand{\Ezero}{\mathbf{E}_0}
\newcommand{\Ezerof}{\mathbf{E}_0^{f}}
\newcommand{\Umat}{\mathbf{U}}
\newcommand{\symd}{\mathbin{\triangle}}
\newcommand{\G}{\mathcal{G}}

\DeclareMathOperator{\sol}{sol}

\DeclareMathOperator{\Card}{Card}

\begin{document}

\title[Sign propagation of principal minors]{Sign Propagation of Principal Minors in Fully Semimonotone Matrices and Stone's Conjecture}

\author*[1]{\fnm{Sajal} \sur{Ghosh}}\email{sajalg@isid.ac.in}

\affil[1]{\orgname{Indian Statistical Institute}, \orgaddress{\street{7 S.J.S. Sansanwal Marg}, \city{New Delhi}, \postcode{110016}, \country{India}}}

\abstract{Stone conjectured in 1981 that every fully semimonotone $\mathbf{Q}_0$-matrix is a $\mathbf{P}_0$-matrix; Murthy and Parthasarathy (1995) proved this for orders up to $4$ and for several special classes. We study the principal minors of fully semimonotone matrices $A\in\Ezerof$ without using $\Qzero$. Let $\G(A)$ be the graph whose vertices are the index sets $\alpha$ with $\det A_{\alpha\alpha}\neq 0$ (including $\alpha=\emptyset$), two vertices being adjacent when they differ in exactly one index. We prove that the sign of the principal minor is constant on every connected component of $\G(A)$, and positive on the component of $\emptyset$. Consequences: an $\Ezerof$-matrix with no vanishing principal minor is a $\mathbf{P}$-matrix, and any $\Ezerof$-matrix satisfying a simple ``descending'' condition on its nonvanishing minors is a $\mathbf{P}_0$-matrix. We show that a sign change across two steps through vanishing minors forces strictly positive coupling in the corresponding principal pivotal transform, and that a principal submatrix of minimal order with negative determinant is an almost-$\mathbf{P}_0$ matrix whose inverse lies in $\Ezerof\cap\Nzero$ and has zero diagonal. Explicit examples show that the component structure can contain negative components, that $\det A>0$ does not force $A\in\Pzero$ in $\Ezerof$, that $\Ezerof$ is not invariant under $A\mapsto A+\varepsilon I$, and that $\Ezerof\cap\Pzero$ contains matrices with $\det A>0$ outside $\Qzero$. Stone's conjecture itself remains open.}

\keywords{Linear complementarity problem, Matrix classes, Fully semimonotone matrices, $\mathbf{Q}_{0}$-matrices, Principal pivotal transform, Principal minors}

\pacs[MSC Classification]{15A48, 90C33}

\maketitle

\section{Introduction}
\label{sec:intro}

For $A \in \Rnn$ and $q \in \Rn$ define
\begin{align}
  F(q,A) &= \{\, z \in \RR^n_+ : Az + q \ge 0 \,\}, \label{eq:F}\\
  S(q,A) &= \{\, z \in F(q,A) : (Az+q)^{t}z = 0 \,\}. \label{eq:S}
\end{align}
The \emph{linear complementarity problem} $(q,A)$ asks for an element of $S(q,A)$; we write $\sol(q,A)=S(q,A)$. For $z\in F(q,A)$ put $w=Az+q\ge 0$; then $z\in S(q,A)$ exactly when $w_iz_i=0$ for all $i$.

Much of the theory of the LCP is organized around matrix classes defined by properties of $\sol(q,A)$ as $q$ varies. The class $\Pmat$ of matrices with all principal minors positive is characterized by unique solvability for every $q$ \cite{SamelsonThrallWesler1958}. Relaxing this property yields, among others, the class $\Qmat$ (a solution exists for every $q$), the class $\Qzero$ (a solution exists for every $q$ with $F(q,A)\neq\emptyset$) and the class $\Pzero$ of matrices with nonnegative principal minors. Cottle and Stone \cite{CottleStone1983} introduced the class $\Umat$ (unique solution for $q$ in the interior of $K(A)$, the union of the complementary cones) and the class $\Ezerof$ of \emph{fully semimonotone} matrices, i.e.\ matrices all of whose principal pivotal transforms are semimonotone, and showed
\[
  \Pmat \subseteq \Umat \subseteq \Ezerof .
\]
Stone \cite{Stone1981} proved $\Umat\cap\Qzero\subseteq\Pzero$ and asked whether this extends to $\Ezerof$:
\begin{conjecture}[Stone \cite{Stone1981}]
  \label{conj:stone}
  $\Ezerof \cap \Qzero \subseteq \Pzero$.
\end{conjecture}
Murthy and Parthasarathy \cite{MurthyParthasarathy1995} (MP) obtained the following results (we cite their numbering). If $A\in\Ezerof\cap\Qzero$ and every proper principal minor of $A$ is nonnegative, then $A\in\Pzero$ (Theorem~3.5). The conjecture holds for $n\le 4$ (Theorem~4.9), for $5\times5$ matrices with positive diagonal (Theorem~4.11), for $6\times 6$ matrices under additional assumptions (Theorem~4.13), and for symmetric, nonnegative, copositive-plus, $Z$- and $E$-matrices (Corollary~3.7). They also show that an $\Ezerof$-matrix of order $n\le3$ with positive diagonal is in $\Pzero$ (Theorem~3.11), that this fails for $n=4$ (Example~3.12), and that fully copositive matrices with positive diagonal are in $\Pzero$ (Theorem~3.17). The class $\Ezerof$ is also studied in \cite{MND2001}. Neogy and Das \cite[Theorem~3.2]{NeogyDas2005} showed that a $\Qzero$-matrix of order $n\ge3$ all of whose PPTs are either in $\Ezero$ or almost copositive, at least one being almost copositive, is a $\Pzero$-matrix. Tucker's classical theorem (see \cite{Tucker1963} and the account in \cite[Section~1]{NeogyDas2005}) states that a matrix all of whose PPTs have positive diagonal is a $\Pmat$-matrix, while nonnegative diagonals do not suffice for $\Pzero$. The conjecture remains open in general.

\paragraph{Contributions.}
MP observe (Remark~1.2) that every principal pivotal transform (PPT) of an $\Ezerof$-matrix has nonnegative diagonal. Translated into minors via the classical formula for the diagonal of a PPT, this is a sign condition on pairs of principal minors that differ in one index. Tucker's theorem uses the same diagonal information when all PPT diagonals are positive. The present note shows that this information can be localized: it yields a sign rule on the graph of nonvanishing principal minors, and we pursue its consequences for $\Ezerof$-matrices, \emph{without using $\Qzero$}.
\begin{itemize}
  \item \emph{Sign adjacency and component signs} (Theorem~\ref{thm:adjacent}, Theorem~\ref{thm:components}). For $A\in\Ezerof$ let $\G(A)$ be the graph on the index sets with nonzero principal minor, adjacency meaning that the sets differ in exactly one index. The sign of the principal minor is constant on each connected component of $\G(A)$, and positive on the component containing $\emptyset$.
  \item \emph{Consequences} (Corollary~\ref{cor:main}). An $\Ezerof$-matrix with no vanishing principal minor is a $\Pmat$-matrix (this case also follows from Tucker's theorem); more generally, an $\Ezerof$-matrix in which every nonvanishing minor has a nonvanishing minor of one order lower (a descending condition) is a $\Pzero$-matrix. In particular, a negative principal minor of an $\Ezerof$-matrix is separated from $\emptyset$ in $\G(A)$.
  \item \emph{Sign change through gaps} (Proposition~\ref{prop:gap}). If two principal minors differ in two indices, both intermediate minors vanish and the signs are opposite, then the PPT has strictly positive coupling in the two corresponding off-diagonal entries.
  \item \emph{Minimal negative minors} (Proposition~\ref{prop:minimal}). A principal submatrix $B$ of minimal order with negative determinant is almost-$\Pzero$, belongs to $\Ezerof$, and $B^{-1}\in\Ezerof\cap\Nzero$ has zero diagonal; all its principal minors of order one less vanish.
  \item \emph{Limits of determinant-based arguments} (Section~\ref{sec:examples}). Explicit examples (i) exhibit negative components of $\G(A)$, with $\det A>0$ and $A\notin\Pzero$; (ii) show that $\Ezerof$ is not invariant under $A\mapsto A+\varepsilon I$; (iii) give a $3\times3$ matrix in $\Ezerof\cap\Pzero$ with $\det A>0$ that is not in $\Qzero$, so that no determinant-sign criterion can substitute for $\Qzero$.
\end{itemize}
These results do not settle Conjecture~\ref{conj:stone}. They do show precisely what a counterexample must look like: its graph $\G(A)$ has a negative component, necessarily disjoint from the component of $\emptyset$.

Section~\ref{sec:prelim} fixes notation and recalls the required facts. Section~\ref{sec:main} contains the main results. Section~\ref{sec:examples} contains the examples and Section~\ref{sec:conclusion} discusses what remains open.

\section{Preliminaries}
\label{sec:prelim}

\subsection{Notation}
\begin{itemize}[leftmargin=2.2cm]
  \item[$\bn$] the index set $\{1,\dots,n\}$; $\bn^{*}$ is the collection of all subsets of $\bn$, including $\emptyset$. (In \cite{MurthyParthasarathy1995} $\bn^*$ denotes the \emph{nonempty} subsets; the empty set is included here for convenience.)
  \item[$\alpha,\beta$] subsets of $\bn$; $\Card(\alpha)$ is the cardinality; $\bar\alpha=\bn\setminus\alpha$; $\alpha\symd\beta$ is the symmetric difference.
  \item[$A_{\alpha\beta}$] submatrix with rows in $\alpha$, columns in $\beta$; $\det A_{\emptyset\emptyset}=1$.
  \item[$e_i$] the $i$th unit vector; $\RR^n_+$ the nonnegative orthant.
\end{itemize}

\subsection{Matrix classes and PPTs}

\begin{definition}
  \label{def:classes}
  Let $A\in\Rnn$.
  \begin{enumerate}[label=(\alph*)]
    \item $A\in\Pmat$ (resp.\ $\Pzero$, $\Nzero$) if $\det A_{\alpha\alpha}>0$ (resp.\ $\ge0$, $\le0$) for all nonempty $\alpha$.
    \item $A$ is \emph{almost-$\Pzero$} if $\det A_{\alpha\alpha}\ge0$ for all nonempty $\alpha\ne\bn$ and $\det A<0$.
    \item $A\in\Qmat$ if $\sol(q,A)\ne\emptyset$ for all $q$; $A\in\Qzero$ if $\sol(q,A)\neq\emptyset$ for all $q$ with $F(q,A)\ne\emptyset$.
    \item $A\in\Ezero$ (\emph{semimonotone}) if for every $z\in\RR^n_+\setminus\{0\}$ there is $k$ with $z_k>0$ and $(Az)_k\ge0$.
  \end{enumerate}
\end{definition}

\begin{definition}[PPT]
  \label{def:PPT}
  If $\det A_{\alpha\alpha}\neq0$ ($\alpha$ is a \emph{valid pivot set}), the \emph{principal pivotal transform} $A^{\alpha}$ is given blockwise by
  \begin{center}
    $(A^\alpha)_{\alpha\alpha}=(A_{\alpha\alpha})^{-1},\quad
    (A^\alpha)_{\alpha\bar\alpha}=-(A_{\alpha\alpha})^{-1}A_{\alpha\bar\alpha},$ \\
    $(A^\alpha)_{\bar\alpha\alpha}=A_{\bar\alpha\alpha}(A_{\alpha\alpha})^{-1},\quad
    (A^\alpha)_{\bar\alpha\bar\alpha}=A_{\bar\alpha\bar\alpha}-A_{\bar\alpha\alpha}(A_{\alpha\alpha})^{-1}A_{\alpha\bar\alpha},$
  \end{center}
  and $A^\emptyset=A$. $A^\alpha$ is not a matrix power. $A$ is \emph{fully semimonotone}, $A\in\Ezerof$, if $A^\alpha\in\Ezero$ for every valid pivot set $\alpha$ (including $\emptyset$) \cite{CottleStone1983}.
\end{definition}

We use the following standard facts.

\begin{lemma}[Inheritance]
  \label{lem:inherit}
  $\Pmat,\Pzero,\Nzero,\Ezero,\Ezerof$ are inherited by principal submatrices \cite[Remark~1.5]{MurthyParthasarathy1995}. If $A\in\Ezerof$ then $A^\alpha\in\Ezerof$ for every valid pivot set $\alpha$.
\end{lemma}
\begin{proof}
  For $\Ezero$: if $z\in\RR^n_+\setminus\{0\}$ has support in $S$, any $k$ with $z_k>0$ lies in $S$ and $(Az)_k=(A_{SS}z_S)_k$. The remaining inheritance statements are routine and are recorded in \cite[Remark~1.5]{MurthyParthasarathy1995}, which also states that $\Ezerof$ is closed under taking PPTs.
\end{proof}

\begin{lemma}[Minors of a PPT]
  \label{lem:ppt-minors}
  Let $\alpha$ be a valid pivot set of $A$ and $\beta\subseteq\bn$. Then
  \[
    \det (A^\alpha)_{\beta\beta}=\frac{\det A_{(\alpha\symd\beta)(\alpha\symd\beta)}}{\det A_{\alpha\alpha}} .
  \]
  In particular, for $j\in\bn$, $(A^\alpha)_{jj}=\det A_{(\alpha\symd\{j\})(\alpha\symd\{j\})}/\det A_{\alpha\alpha}$.
\end{lemma}
\begin{proof}
  Both sides are rational functions of the entries of $A$ on the open set $\{\det A_{\alpha\alpha}\ne0\}$, so it suffices to verify the identity on the dense subset where moreover $\det A_{QQ}\ne0$, with $Q=\alpha\setminus\beta$. Put $P=\alpha\cap\beta$ and $R=\beta\setminus\alpha$; then $\alpha=Q\cup P$, $\beta=P\cup R$ and $\alpha\symd\beta=Q\cup R$. Pivoting first on $Q$ and then on $P$ gives $A^\alpha=\mathrm{ppt}(A^Q,P)$ \cite[Theorem~3.2]{Tsatsomeros2000}, and $(A^Q)_{PP}$ is the Schur complement of $A_{QQ}$ in $A_{\alpha\alpha}$, so $\det(A^Q)_{PP}=\det A_{\alpha\alpha}/\det A_{QQ}\ne0$. The blocks of a PPT with respect to $P$ restricted to $\beta$ involve only rows and columns in $\beta$, hence $(A^\alpha)_{\beta\beta}=\mathrm{ppt}\bigl((A^Q)_{\beta\beta},P\bigr)$. By \cite[Theorem~3.8(i)]{Tsatsomeros2000},
  \[
    \det(A^\alpha)_{\beta\beta}=\frac{\det (A^Q)_{RR}}{\det (A^Q)_{PP}} .
  \]
  Since $R\subseteq\bar Q$, $(A^Q)_{RR}$ is the Schur complement of $A_{QQ}$ in $A_{(Q\cup R)(Q\cup R)}$, so $\det(A^Q)_{RR}=\det A_{(Q\cup R)(Q\cup R)}/\det A_{QQ}$. Dividing yields the claim. (For $\beta=\{j\}$ the formula is also immediate from the cofactor formula when $j\in\alpha$ and from the Schur determinant formula when $j\notin\alpha$.)
\end{proof}

\begin{lemma}[Elementary facts on $\Ezero$]
  \label{lem:E0}
  \begin{enumerate}[label=(\roman*)]
    \item If $A\in\Ezero$ then $a_{jj}\ge0$ for all $j$.
    \item A matrix $\begin{pmatrix}0&b\\c&0\end{pmatrix}$ belongs to $\Ezero$ if and only if $b\ge0$ or $c\ge0$.
  \end{enumerate}
\end{lemma}
\begin{proof}
  (i) Take $z=e_j$; the only admissible index is $j$, so $a_{jj}=(Ae_j)_j\ge0$.
  (ii) For $z=(z_1,z_2)^t\ge0$, $z\ne0$ we have $Az=(bz_2,cz_1)^t$. If both $z_1,z_2>0$ we need $b\ge0$ or $c\ge0$; if $z=(z_1,0)^t$ or $(0,z_2)^t$ the index with positive $z_k$ gives value $0$. Hence the condition is necessary (take $z=(1,1)^t$) and sufficient.
\end{proof}

\begin{remark}
  Lemma~\ref{lem:E0}(i) applied to $A^\alpha$ is Remark~1.2 of \cite{MurthyParthasarathy1995}, and the identity of Lemma~\ref{lem:ppt-minors} for $\beta=\{j\}$ appears implicitly in the proofs of \cite[Theorems~3.17 and~4.11]{MurthyParthasarathy1995}.
\end{remark}

\section{Sign propagation of principal minors}
\label{sec:main}

\begin{theorem}[Sign adjacency]
  \label{thm:adjacent}
  Let $A\in\Ezerof$, let $\alpha$ be a valid pivot set and $j\in\bn$. Then
  \[
    \det A_{\alpha\alpha}\cdot\det A_{(\alpha\symd\{j\})(\alpha\symd\{j\})}\ \ge\ 0 .
  \]
\end{theorem}
\begin{proof}
  Since $A^\alpha\in\Ezero$, Lemma~\ref{lem:E0}(i) gives $(A^\alpha)_{jj}\ge0$. By Lemma~\ref{lem:ppt-minors},
  $0\le (A^\alpha)_{jj}=\det A_{(\alpha\symd\{j\})(\alpha\symd\{j\})}/\det A_{\alpha\alpha}$. Multiply by $(\det A_{\alpha\alpha})^2>0$.
\end{proof}

\begin{definition}[Nonvanishing graph]
  \label{def:graph}
  For $A\in\Rnn$ let $\G(A)$ be the graph with vertex set $V(A)=\{\alpha\subseteq\bn:\det A_{\alpha\alpha}\ne0\}$ (so $\emptyset\in V(A)$), in which $\alpha,\alpha'$ are adjacent if and only if $\Card(\alpha\symd\alpha')=1$.
\end{definition}

\begin{theorem}[Constant sign on components]
  \label{thm:components}
  Let $A\in\Ezerof$. Then $\det A_{\alpha\alpha}$ has the same sign for all $\alpha$ in a given connected component of $\G(A)$. On the component of $\emptyset$ all principal minors are positive.
\end{theorem}
\begin{proof}
  If $\alpha,\alpha'\in V(A)$ are adjacent, Theorem~\ref{thm:adjacent} gives $\det A_{\alpha\alpha}\det A_{\alpha'\alpha'}\ge0$, and the product is nonzero, hence positive. Propagate along paths. Since $\det A_{\emptyset\emptyset}=1>0$, the second claim follows.
\end{proof}

\begin{corollary}
  \label{cor:main}
  Let $A\in\Ezerof$.
  \begin{enumerate}[label=(\alph*)]
    \item (Nondegenerate case) If every principal minor of $A$ is nonzero, then $A\in\Pmat$.
    \item (Descending condition) Suppose that for every nonempty $\alpha\in V(A)$ there is $i\in\alpha$ with $\alpha\setminus\{i\}\in V(A)$. Then $A\in\Pzero$.
    \item If $\det A_{\alpha\alpha}<0$ for some $\alpha$, then $\alpha$ is not in the component of $\emptyset$ in $\G(A)$. In particular $A$ has a vanishing principal minor, and no chain $\emptyset=\alpha_0,\alpha_1,\dots,\alpha_k=\alpha$ of nonvanishing minors with $\Card(\alpha_{i-1}\symd\alpha_i)=1$ exists.
  \end{enumerate}
\end{corollary}
\begin{proof}
  (a) Every $\alpha=\{i_1,\dots,i_m\}$ is joined to $\emptyset$ by the path $\emptyset,\{i_1\},\{i_1,i_2\},\dots,\alpha$ in $\G(A)$; apply Theorem~\ref{thm:components}.
  (b) We show by induction on $\Card(\alpha)$ that every $\alpha\in V(A)$ lies in the component of $\emptyset$. For $\alpha=\emptyset$ this is trivial. If $\alpha\neq\emptyset$, choose $i$ with $\alpha\setminus\{i\}\in V(A)$; it is adjacent to $\alpha$ and lies in the component of $\emptyset$ by induction. Hence all nonzero principal minors are positive.
  (c) is the contrapositive of the positivity statement of Theorem~\ref{thm:components}; if all principal minors were nonzero, (a) would give $A\in\Pmat$.
\end{proof}

Consequently, Conjecture~\ref{conj:stone} holds for matrices satisfying (a) or (b), even without the hypothesis $\Qzero$.

\begin{remark}[Relation to Tucker's theorem]
  Part (a) is also an immediate consequence of Tucker's theorem: if $A\in\Ezerof$ has all principal minors nonzero, then by Lemma~\ref{lem:ppt-minors} every PPT has diagonal entries $\det A_{(\alpha\symd\{j\})(\alpha\symd\{j\})}/\det A_{\alpha\alpha}$, which are nonzero, and nonnegative by Theorem~\ref{thm:adjacent}, hence positive, so $A\in\Pmat$. Theorem~\ref{thm:components} localizes this argument: only chains of nonvanishing minors are needed, and it determines the sign of minors in components not containing $\emptyset$. Nonnegativity of all PPT diagonals alone does not give $\Pzero$, as $J$ in Example~\ref{ex:J} shows.
\end{remark}

\begin{proposition}[Sign change across a gap]
  \label{prop:gap}
  Let $A\in\Ezerof$, let $\alpha$ be a valid pivot set and let $i\ne j$ be indices such that
  \[
    \det A_{(\alpha\symd\{i\})(\alpha\symd\{i\})}=\det A_{(\alpha\symd\{j\})(\alpha\symd\{j\})}=0 .
  \]
  If $\det A_{\alpha\alpha}\cdot\det A_{(\alpha\symd\{i,j\})(\alpha\symd\{i,j\})}<0$, then $(A^\alpha)_{ij}>0$ and $(A^\alpha)_{ji}>0$.
\end{proposition}
\begin{proof}
  Put $M=A^\alpha\in\Ezero$. By Lemma~\ref{lem:ppt-minors}, $m_{ii}=m_{jj}=0$ and
  \[
    \frac{\det A_{(\alpha\symd\{i,j\})(\alpha\symd\{i,j\})}}{\det A_{\alpha\alpha}}=\det M_{\{i,j\}\{i,j\}}=-m_{ij}m_{ji}.
  \]
  The hypothesis gives $m_{ij}m_{ji}>0$. The principal submatrix $M_{\{i,j\}\{i,j\}}=\begin{pmatrix}0&m_{ij}\\m_{ji}&0\end{pmatrix}$ lies in $\Ezero$ (Lemma~\ref{lem:inherit}), so by Lemma~\ref{lem:E0}(ii) $m_{ij}\ge0$ or $m_{ji}\ge0$. Since $m_{ij}m_{ji}>0$ both are positive.
\end{proof}

\begin{proposition}[Minimal negative minors]
  \label{prop:minimal}
  Let $A\in\Ezerof$ have a negative principal minor, let $m$ be the smallest cardinality of an index set $\alpha$ with $\det A_{\alpha\alpha}<0$, and put $B=A_{\alpha\alpha}$. Then:
  \begin{enumerate}[label=(\roman*)]
    \item $m\ge2$;
    \item every principal minor of $A$ (and of $B$) of order $<m$ is nonnegative, and $\det A_{(\alpha\setminus\{j\})(\alpha\setminus\{j\})}=0$ for every $j\in\alpha$; hence $B^{-1}$ exists and has zero diagonal;
    \item $B$ is an almost-$\Pzero$ matrix in $\Ezerof$, and $B^{-1}\in\Ezerof\cap\Nzero$;
    \item if $m=2$, then $B=\begin{pmatrix}0&b\\c&0\end{pmatrix}$ with $b,c>0$;
    \item $B^{-1}$ is nonsingular and lies in $\Ezero\cap\Nzero$, so by \cite[Theorem~3.4]{MurthyParthasarathy1995} some principal rearrangement of $B^{-1}$ has the block form $\begin{pmatrix}U&X\\Y&V\end{pmatrix}$ with $U,V$ strictly upper triangular and nonpositive (of orders $\ge1$) and $X,Y\ge0$.
  \end{enumerate}
\end{proposition}
\begin{proof}
  (i) If $m=1$ then $a_{jj}<0$ for some $j$, contradicting Lemma~\ref{lem:E0}(i) for $A^\emptyset=A\in\Ezero$.

  (ii) The nonnegativity is the definition of minimality. Fix $j\in\alpha$. As $m\ge2$ we have $\alpha\setminus\{j\}\ne\emptyset$, so $\det A_{(\alpha\setminus\{j\})(\alpha\setminus\{j\})}\ge0$. Theorem~\ref{thm:adjacent} with $\det A_{\alpha\alpha}<0$ forces it to be $\le0$; hence it is $0$. By the cofactor formula $(B^{-1})_{jj}=0$.

  (iii) $B\in\Ezerof$ by Lemma~\ref{lem:inherit}, and its proper nonempty principal minors are nonnegative by minimality, while $\det B<0$; so $B$ is almost-$\Pzero$. $B^{-1}=B^{\alpha}$ (pivoting on the full index set) belongs to $\Ezerof$ by Lemma~\ref{lem:inherit}. For nonempty $\beta\subseteq\alpha$, Lemma~\ref{lem:ppt-minors} (with the full index set as pivot) gives $\det (B^{-1})_{\beta\beta}=\det B_{\bar\beta\bar\beta}/\det B$ where $\bar\beta$ is the complement of $\beta$ in $\alpha$. If $\beta=\alpha$ this equals $1/\det B<0$; otherwise $\bar\beta$ is a nonempty set of cardinality $<m$ and $\det B_{\bar\beta\bar\beta}\ge0$, so the quotient is $\le0$. Thus $B^{-1}\in\Nzero$.

  (iv) By (ii), $b_{11}=b_{22}=0$, so $B=\begin{pmatrix}0&b\\c&0\end{pmatrix}$ with $-bc<0$. By Lemma~\ref{lem:E0}(ii), $b\ge0$ or $c\ge0$, and $bc>0$ gives $b,c>0$.

  (v) $B^{-1}\in\Ezerof\cap\Nzero$ by (iii), hence in $\Ezero\cap\Nzero$; it is nonsingular, and the stated structure is \cite[Theorem~3.4]{MurthyParthasarathy1995} (see also \cite[Theorem~2.7]{NeogyDas2005}).
\end{proof}

\begin{remark}[Relation to \cite{MurthyParthasarathy1995}]
  \label{rmk:MP}
  Proposition~\ref{prop:minimal}(iii) places a minimal negative principal submatrix exactly in the setting in which \cite[Theorem~3.5]{MurthyParthasarathy1995} is proved: $B^{-1}\in\Ezerof\cap\Nzero$, to which \cite[Theorem~3.4]{MurthyParthasarathy1995} gives a block structure. Applying Theorem~3.5 of that paper to $B$ shows $B\notin\Qzero$. Hence a counterexample $A$ to Conjecture~\ref{conj:stone} would be a $\Qzero$-matrix with a proper principal submatrix $B$ of the above form that is not in $\Qzero$ (proper, since $B=A$ would give $A\notin\Qzero$). Note that $\Qzero$ is not inherited by principal submatrices \cite[Remark~1.5]{MurthyParthasarathy1995}; $\Qzero$ therefore has to be used on $A$ itself, and the results of this section make no use of it.
\end{remark}

\section{Examples}
\label{sec:examples}

\begin{example}
  \label{ex:J}
  Let $J=\begin{pmatrix}0&1\\1&0\end{pmatrix}$. Then $J\in\Ezerof$, $\det J=-1$, and $J\notin\Qzero$.

  The valid pivot sets are $\emptyset$ and $\{1,2\}$ (the diagonal entries vanish), and $J^\emptyset=J$, $J^{\{1,2\}}=J^{-1}=J$. For $z\in\RR^2_+\setminus\{0\}$ choose $k$ with $z_k>0$; then $(Jz)_k=z_{3-k}\ge0$, so $J\in\Ezero$ and hence $J\in\Ezerof$.

  For $q=(1,-1)^t$, $z=(1,0)^t\in F(q,J)$. For any $z\in\RR^2_+$, $w=Jz+q=(z_2+1,\ z_1-1)^t$; $w_1>0$ forces $z_1=0$, and then $w_2=-1<0$. So $\sol(q,J)=\emptyset$.

  Here $\G(J)$ has the two isolated vertices $\emptyset$ and $\{1,2\}$, with signs $+$ and $-$. This also illustrates Proposition~\ref{prop:gap}: $\det J<0$ and $J_{12}=J_{21}=1>0$.
\end{example}

\begin{lemma}[Direct sums]
  \label{lem:dsum}
  If $B_1,B_2\in\Ezerof$ then $B_1\oplus B_2\in\Ezerof$.
\end{lemma}
\begin{proof}
  An index set $\alpha=\alpha_1\cup\alpha_2$ (first and second block) is a valid pivot set of $B_1\oplus B_2$ iff $\alpha_1,\alpha_2$ are valid for $B_1,B_2$, and then $(B_1\oplus B_2)^\alpha=B_1^{\alpha_1}\oplus B_2^{\alpha_2}=:C_1\oplus C_2$ with $C_i\in\Ezero$. Let $z=(z^{(1)},z^{(2)})\ge0$, $z\neq0$. If $z^{(1)}\ne0$ pick $k$ in the first block with $z_k>0$ and $(C_1z^{(1)})_k\ge0$; then $((C_1\oplus C_2)z)_k=(C_1z^{(1)})_k\ge0$. If $z^{(1)}=0$ argue in the second block.
\end{proof}

\begin{example}[A negative component with $\det A>0$]
  \label{ex:JJ}
  Let $B=J\oplus J\in\RR^{4\times4}$. Then $B\in\Ezerof$ (Example~\ref{ex:J} and Lemma~\ref{lem:dsum}), $\det B=1>0$, $B\notin\Pzero$ since $\det B_{\{1,2\}\{1,2\}}=-1$, and $B\notin\Qzero$.

  The nonzero principal minors are those on $\emptyset,\{1,2\},\{3,4\},\{1,2,3,4\}$, with values $1,-1,-1,1$ (all other principal minors vanish, since they contain a zero diagonal entry, or a $0$ row/column block). No two of these index sets differ in exactly one index, so $\G(B)$ has four isolated vertices; the components $\{1,2\}$ and $\{3,4\}$ are negative. Thus $\det B>0$ coexists with negative components, and the sign of the determinant of $A$ alone says nothing about $\Pzero$-membership.

  For $q=(1,-1,0,0)^t$ we have $z=(1,0,0,0)^t\in F(q,B)$ while $w=Bz+q=(z_2+1,z_1-1,z_4,z_3)^t$ forces $z_1=0$ and $w_2=-1<0$ for every $z\in\RR^4_+$ with $w_1z_1=0$. Hence $\sol(q,B)=\emptyset$ and $B\notin\Qzero$.
\end{example}

\begin{proposition}[$\Ezerof$ is not invariant under $A\mapsto A+\varepsilon I$]
  \label{prop:perturb}
  Let $A\in\{J,\,J\oplus J\}$ and $0<\varepsilon<1$. Then $A\in\Ezerof$ but $A+\varepsilon I\notin\Ezerof$.
\end{proposition}
\begin{proof}
  The $(1,1)$ entry of $A+\varepsilon I$ is $\varepsilon\ne0$, so $\{1\}$ is a valid pivot set, and $\det(A+\varepsilon I)_{\{1,2\}\{1,2\}}=\varepsilon^2-1$. By Lemma~\ref{lem:ppt-minors}, $\bigl((A+\varepsilon I)^{\{1\}}\bigr)_{22}=(\varepsilon^2-1)/\varepsilon<0$, so $(A+\varepsilon I)^{\{1\}}\notin\Ezero$ by Lemma~\ref{lem:E0}(i).
\end{proof}

\begin{remark}
  For $B\in\Ezero$ and $\varepsilon>0$ one has $B+\varepsilon I\in\Ezero$: pick $k$ with $z_k>0$ and $(Bz)_k\ge0$; then $((B+\varepsilon I)z)_k>0$. By Proposition~\ref{prop:perturb} this does not pass to $\Ezerof$, since PPTs of $A+\varepsilon I$ are not shifts of PPTs of $A$. Thus a perturbation argument ($A\mapsto A+\varepsilon I$, then $\varepsilon\to0^+$) is unavailable for $\Ezerof$, whereas for $A\in\Pzero$ one has $A+\varepsilon I\in\Pmat\subseteq\Ezerof$ since $\det(A_{\alpha\alpha}+\varepsilon I)=\sum_{S\subseteq\alpha}\varepsilon^{\Card(\alpha)-\Card(S)}\det A_{SS}$ has positive leading coefficient and nonnegative remaining coefficients.
\end{remark}

\begin{example}[$\Ezerof\cap\Pzero$ with $\det A>0$ and $A\notin\Qzero$]
  \label{ex:M}
  Let
  \[
    M=\begin{pmatrix}0&2&-2\\0&1&0\\2&2&0\end{pmatrix}.
  \]
  Then $M\in\Ezerof\cap\Pzero$, $\det M=4$, and $M\notin\Qzero$.

  \smallskip\noindent\emph{Principal minors.} The diagonal entries are $0,1,0$; the order-two minors on $\{1,2\},\{1,3\},\{2,3\}$ are $0,4,0$; and $\det M=4$. Hence $M\in\Pzero$, and the valid pivot sets are $\emptyset,\{2\},\{1,3\},\{1,2,3\}$. The graph $\G(M)$ has two components, $\{\emptyset,\{2\}\}$ (minors $1,1$) and $\{\{1,3\},\{1,2,3\}\}$ (minors $4,4$), both positive.

  \smallskip\noindent\emph{$M\in\Ezerof$.} Direct computation from Definition~\ref{def:PPT} gives $M^\emptyset=M^{\{2\}}=M$ and $M^{\{1,3\}}=M^{\{1,2,3\}}=M^{-1}=\begin{pmatrix}0&-1&1/2\\0&1&0\\-1/2&1&0\end{pmatrix}$. For $z\ge0$, $z\ne0$: $Mz=(2z_2-2z_3,\ z_2,\ 2z_1+2z_2)^t$. If $z_2>0$ take $k=2$. If $z_2=0$ and $z_1>0$: take $k=1$ when $z_3=0$ (value $0$) and $k=3$ when $z_3>0$ (value $2z_1>0$). If $z_1=z_2=0$, then $z_3>0$ and $k=3$ gives $0$. Also $M^{-1}z=(-z_2+\tfrac12z_3,\ z_2,\ -\tfrac12z_1+z_2)^t$. If $z_2>0$ take $k=2$. If $z_2=0$: for $z_3>0,z_1=0$ take $k=3$ (value $0$); for $z_3>0,z_1>0$ take $k=1$ (value $z_3/2>0$); for $z_3=0$ we have $z_1>0$ and $k=1$ gives $0$. So $M,M^{-1}\in\Ezero$.

  \smallskip\noindent\emph{$M\notin\Qzero$.} Let $q=(-3,-1,-3)^t$. Then $z=(0,\tfrac32,0)^t\in F(q,M)$ since $Mz+q=(0,\tfrac12,0)^t$. If $z\in\RR^3_+$ solves $(q,M)$, with $w=Mz+q=(2z_2-2z_3-3,\ z_2-1,\ 2z_1+2z_2-3)^t$, then $w_2\ge0$ gives $z_2\ge1>0$, so $w_2=0$, i.e.\ $z_2=1$, whence $w_1=-1-2z_3<0$, a contradiction.
\end{example}

\noindent The direct sum $M\oplus I_{n-3}$ gives such examples in every order $n\ge3$ (Lemma~\ref{lem:dsum}, $I\in\Pmat\subseteq\Ezerof$; for $q=(q_M,0)$ the problem has no solution). The zero matrix lies in $\Ezerof\cap\Qzero$ and is singular; so $\det A>0$ is neither implied by, nor does it imply, membership in $\Ezerof\cap\Qzero$ or in $\Qzero$ alone.

\section{Minimal counterexamples}
\label{sec:mincex}

In this section we use $\Qzero$. Call $A$ a \emph{minimal counterexample} to Conjecture~\ref{conj:stone} if $A\in\Ezerof\cap\Qzero$, $A\notin\Pzero$, and every matrix in $\Ezerof\cap\Qzero$ of smaller order is in $\Pzero$. By \cite{MurthyParthasarathy1995} such an $A$ has order $n\ge5$. We write $[i\in\delta]$ for $1$ if $i\in\delta$ and $0$ otherwise.

\begin{proposition}[Negative row property]
  \label{prop:negrow}
  Let $A$ be a minimal counterexample of order $n$, let $\beta$ be a valid pivot set and $i\in\bn$. Suppose there is an index set $\delta\ne\beta$ with $[i\in\delta]=[i\in\beta]$ and $\det A_{\delta\delta}\cdot\det A_{\beta\beta}<0$. Then row $i$ of $A^\beta$ contains a negative entry.
\end{proposition}
\begin{proof}
  $M=A^\beta$ belongs to $\Ezerof\cap\Qzero$ \cite[Remark~1.5]{MurthyParthasarathy1995}. Suppose $M_{i\cdot}\ge0$. By \cite[Corollary~2.3]{MurthyParthasarathy1995}, $M_{\kappa\kappa}\in\Qzero$ for $\kappa=\bn\setminus\{i\}$; also $M_{\kappa\kappa}\in\Ezerof$ by Lemma~\ref{lem:inherit}. Since $M_{\kappa\kappa}$ has order $n-1$, minimality gives $M_{\kappa\kappa}\in\Pzero$. Put $\gamma=\beta\symd\delta$. Then $\gamma\ne\emptyset$ and $i\notin\gamma$, so $\gamma\subseteq\kappa$, and by Lemma~\ref{lem:ppt-minors}
  \[
    \det M_{\gamma\gamma}=\frac{\det A_{(\beta\symd\gamma)(\beta\symd\gamma)}}{\det A_{\beta\beta}}=\frac{\det A_{\delta\delta}}{\det A_{\beta\beta}}<0,
  \]
  contradicting $M_{\kappa\kappa}\in\Pzero$.
\end{proof}

\begin{corollary}
  \label{cor:negrow}
  Let $A$ be a minimal counterexample and let $\alpha$ be any index set with $\det A_{\alpha\alpha}<0$.
  \begin{enumerate}[label=(\alph*)]
    \item Every row $i\notin\alpha$ of $A$, and of $A^\alpha$, contains a negative entry.
    \item Every index $i$ for which row $i$ of $A$ is nonnegative lies in every index set with negative principal minor. In particular $A$ has no zero row.
  \end{enumerate}
\end{corollary}
\begin{proof}
  (a) For $A=A^\emptyset$ apply Proposition~\ref{prop:negrow} with $\beta=\emptyset$ and $\delta=\alpha$ (note $i\notin\alpha$ and $\det A_{\emptyset\emptyset}=1>0$). For $A^\alpha$ take $\beta=\alpha$ and $\delta=\emptyset$.
  (b) is the contrapositive of (a) for $A$. If row $i$ were zero, every principal minor on a set containing $i$ would vanish, so $i$ could not lie in a negative index set; but $A$ has a negative principal minor.
\end{proof}

\begin{proposition}[Negative $2\times2$ minors: structure of rows and columns]
  \label{prop:twoblock}
  Let $A\in\Ezerof$ have $a_{11}=a_{22}=0$ and $a_{12}=b>0$, $a_{21}=c>0$ (a negative $2\times2$ principal minor; by Proposition~\ref{prop:minimal}(iv) every negative $2\times 2$ minor has this form up to a principal rearrangement). Then for every $k\ge3$,
  \[
    a_{k1}<0\ \Longrightarrow\ a_{1k}=0, \qquad a_{k2}<0\ \Longrightarrow\ a_{2k}=0 .
  \]
  If moreover $A\in\Qzero$, then there are indices $k,l\ge3$ with $a_{k1}<0$ and $a_{l2}<0$, and hence $a_{1k}=0$ and $a_{2l}=0$.
\end{proposition}
\begin{proof}
  A $2\times2$ matrix $\begin{pmatrix}0&p\\r&d\end{pmatrix}$ in $\Ezero$ satisfies $p\ge0$ or $r\ge0$: if $p,r<0$, then $z=(1,t)^t$ with small $t>0$ gives $(Mz)_1=pt<0$ and $(Mz)_2=r+dt<0$.

  Fix $k\ge3$ and put $g=a_{1k}$, $h=a_{2k}$, $x=a_{k1}$, $y=a_{k2}$. The principal submatrix $A'$ on $\{1,2,k\}$ lies in $\Ezerof$. Its principal submatrix on $\{1,k\}$ is $\begin{pmatrix}0&g\\x&a_{kk}\end{pmatrix}\in\Ezero$, so $x<0$ implies $g\ge0$. Now $\{1,2\}$ is a valid pivot set of $A'$, with $B^{-1}=\begin{pmatrix}0&1/c\\1/b&0\end{pmatrix}$, and in $N=(A')^{\{1,2\}}\in\Ezero$ we have $N_{22}=0$, $N_{2k}=-g/b$, $N_{k2}=x/c$. The principal submatrix of $N$ on $\{2,k\}$ is $\begin{pmatrix}0&-g/b\\x/c&N_{kk}\end{pmatrix}\in\Ezero$, so $x<0$ implies $-g/b\ge0$, i.e.\ $g\le0$. Hence $x<0$ implies $g=0$. Interchanging the indices $1$ and $2$ gives the second implication.

  If $A\in\Qzero$, \cite[Theorem~2.9]{MurthyParthasarathy1995} applied to the zero diagonal entries $a_{11}$ (with $a_{12}>0$) and $a_{22}$ (with $a_{21}>0$) yields negative entries $a_{k1}$ and $a_{l2}$; they lie outside rows $1,2$ because $a_{21},a_{12}>0$ and $a_{11}=a_{22}=0$.
\end{proof}

\begin{proposition}[Solutions to deletion problems]
  \label{prop:delete}
  Let $A$ be a minimal counterexample and let $\det A_{\alpha\alpha}<0$. Then for every $k\notin\alpha$ there is $u\in S(e_k,A)$ with $(Au)_k=-1$. If $\operatorname{supp}u=\gamma$ and $k\in\gamma$, then $\det A_{\gamma\gamma}=0$.
\end{proposition}
\begin{proof}
  $A_{(\bn\setminus\{k\})(\bn\setminus\{k\})}$ contains $A_{\alpha\alpha}$, so it is not in $\Pzero$; it is in $\Ezerof$, hence not in $\Qzero$ by minimality. The first assertion is then \cite[Theorem~2.2]{MurthyParthasarathy1995}. If $k\in\gamma$ and $\det A_{\gamma\gamma}\ne0$, then complementarity gives $A_{\gamma\gamma}u_\gamma=-e_k$ with $u_\gamma>0$, so $(A_{\gamma\gamma}^{-1})_{kk}<0$. By Lemma~\ref{lem:ppt-minors} this is the $(k,k)$ entry of $A^\gamma$, contradicting $A^\gamma\in\Ezero$ (Lemma~\ref{lem:E0}(i)).
\end{proof}

\subsection*{Negative minors of order two and three}

The following observation is the main new tool of this section. It uses only $\Ezero$ for a matrix and for its PPTs.

\begin{lemma}[Support lemma]
  \label{lem:support}
  Let $M\in\Ezero$, let $u\in\RR^n_+\setminus\{0\}$ and $v=Mu+q\ge0$ with $u^tv=0$. Then $q_k\le0$ for some $k\in\operatorname{supp}(u)$.
\end{lemma}
\begin{proof}
  For $k\in\operatorname{supp}(u)$ we have $v_k=0$, i.e.\ $(Mu)_k=-q_k$. Since $M\in\Ezero$ there is $k\in\operatorname{supp}(u)$ with $(Mu)_k\ge0$, so $q_k\le0$.
\end{proof}

Let $A\in\Ezerof$ and let $\beta$ be a valid pivot set. If $(w,z)$ solves the LCP $(q,A)$, i.e.\ $w=Az+q\ge0$, $z\ge0$, $w^tz=0$, then, with $M=A^\beta$, the pair $v=(z_\beta,w_{\bar\beta})$, $u=(w_\beta,z_{\bar\beta})$ solves $(q^\beta,M)$ where $q^\beta_\beta=-A_{\beta\beta}^{-1}q_\beta$ and $q^\beta_{\bar\beta}=q_{\bar\beta}-A_{\bar\beta\beta}A_{\beta\beta}^{-1}q_\beta$. Applying Lemma~\ref{lem:support} in several frames $\beta$ to the same solution produces constraints that can be contradictory.

\begin{theorem}[Negative $2\times2$ minors]
  \label{thm:order2}
  Let $A\in\Ezerof\cap\Qzero$ (of any order $n$). Then every $2\times2$ principal minor of $A$ is nonnegative.
\end{theorem}
\begin{proof}
  Suppose a $2\times2$ principal minor is negative. Order-one minors are nonnegative (Lemma~\ref{lem:E0}(i)), so by Proposition~\ref{prop:minimal}(iv) after a principal rearrangement $a_{11}=a_{22}=0$, $a_{12}=b>0$, $a_{21}=c>0$. Put $\alpha=\{1,2\}$, $B=A_{\alpha\alpha}$, so $B^{-1}=\begin{pmatrix}0&1/c\\1/b&0\end{pmatrix}$. For $t>0$ let $q=(1,-1,t,\dots,t)^t$. For large $t$:
  \begin{itemize}
    \item $F(q,A)\ne\emptyset$: take $z=(1/c,0,\dots,0)^t$; then $Az+q=(1,0,(a_{k1}/c+t)_{k\ge3})\ge0$.
    \item $q^\alpha=(1/c,\,-1/b,\,(t+a_{k1}/c-a_{k2}/b)_{k\ge3})$, whose entries outside index $2$ are positive.
  \end{itemize}
  Suppose $(w,z)$ solves $(q,A)$. Then $z\ne0$ (as $q\not\ge0$), and Lemma~\ref{lem:support} for $M=A$ gives $z_2>0$, because only $q_2\le0$. For $M=A^\alpha$ the vector $u=(w_1,w_2,z_3,\dots,z_n)$ is nonzero (as $q^\alpha_2<0$) and Lemma~\ref{lem:support} gives $w_2>0$. This contradicts $z_2w_2=0$, so $S(q,A)=\emptyset$ and $A\notin\Qzero$.
\end{proof}

\begin{theorem}[Minimal negative minors of order three]
  \label{thm:order3}
  Let $A\in\Ezerof\cap\Qzero$ (of any order $n$). Then $A$ has no principal minor of order three that is negative while all principal minors of order $\le2$ are nonnegative. Consequently, if $A\in\Ezerof\cap\Qzero$ has a negative principal minor, then the smallest order of a negative principal minor is at least $4$.
\end{theorem}
\begin{proof}
  Let $\det A_{\alpha\alpha}<0$ with $\Card(\alpha)=3$ and all principal minors of $A$ of order $\le2$ nonnegative; put $B=A_{\alpha\alpha}$ and $N=B^{-1}$. By Proposition~\ref{prop:minimal}(ii),(iii), $N\in\Ezero\cap\Nzero$ with zero diagonal and $\det N<0$. For a pair $i\ne j$, $\Nzero$ gives $n_{ij}n_{ji}\ge0$ and Lemma~\ref{lem:E0}(ii) (applied to the $2\times2$ principal submatrix of $N$) gives $n_{ij}\ge0$ or $n_{ji}\ge0$; hence if $n_{ij}<0$ then $n_{ji}=0$. Writing $\det N=n_{12}n_{23}n_{31}+n_{13}n_{32}n_{21}$ and relabelling, we may assume $n_{12}n_{23}n_{31}<0$ and $n_{13}n_{32}n_{21}=0$ is the other term. All three factors of the negative product are nonzero and an odd number is negative. If all three were negative, their reverses vanish, and $z=(1,1,1)^t$ gives $Nz<0$, contradicting $N\in\Ezero$. So exactly one is negative; cyclically relabelling it is $n_{23}=:v<0$, with $n_{12}=:x_1>0$, $n_{31}=:y_2>0$, $n_{32}=0$, $n_{13}=:x_2\ge0$, $n_{21}=:y_1\ge0$:
  \[
    N=\begin{pmatrix}0&x_1&x_2\\y_1&0&v\\y_2&0&0\end{pmatrix}.
  \]
  Let $q_B=(-1,1,1)^t$ and $q=(q_B,t,\dots,t)$ with $t$ large. We have $-Nq_B=(-(x_1+x_2),\ y_1-v,\ y_2)$, so the only nonpositive entry of both $q_B$ and $-Nq_B$ is the first. Feasibility: for $w=(0,w_2,0)$ with $w_2\ge (x_1+x_2)/x_1$ put $z_B=N(w-q_B)=\bigl(x_1w_2-x_1-x_2,\ y_1-v,\ y_2\bigr)\ge0$; then $Bz_B+q_B=w\ge0$, and $z=(z_B,0)$ is feasible for $(q,A)$ once $t$ is large. Suppose $(w,z)$ solves $(q,A)$. Lemma~\ref{lem:support} for $M=A$ (only $q_1\le0$) gives $z_1>0$. For $M=A^\alpha$, $q^\alpha_\alpha=-Nq_B$ and the entries of $q^\alpha$ outside $\alpha$ are $t-A_{\bar\alpha\alpha}Nq_B>0$ for large $t$; the vector $u=(w_\alpha,z_{\bar\alpha})\ne0$ (since $q^\alpha_1<0$), so Lemma~\ref{lem:support} gives $w_1>0$, contradicting $z_1w_1=0$. Hence $A\notin\Qzero$.
\end{proof}

\begin{corollary}
  \label{cor:n5}
  Let $A\in\Ezerof\cap\Qzero$ of order $n\ge4$ have a negative principal minor, and let $m$ be the smallest order of a negative principal minor. Then $m\ge4$ and, by \cite[Theorem~3.5]{MurthyParthasarathy1995}, $m\le n-1$. In particular Conjecture~\ref{conj:stone} holds for $n\le4$, and a counterexample of order $5$ would have all principal minors of order $\le3$ nonnegative and a negative principal minor of order exactly $4$.
\end{corollary}

\begin{proposition}[A criterion for $A\notin\Qzero$]
  \label{prop:criterion}
  Let $A\in\Ezerof$, let $\alpha$ be an index set with $B=A_{\alpha\alpha}$ nonsingular, and put $N=B^{-1}$. Suppose there are $i\in\alpha$, $q_B\in\RR^{\alpha}$ and $w\in\RR^{\alpha}_+$ such that\\
  \begin{center}
    $
    (q_B)_i<0,\quad (q_B)_j>0\ (j\ne i),\quad (Nq_B)_i\ge0,$ \\ 
    $(Nq_B)_j<0 \quad (j\ne i),\quad N(w-q_B)\ge0 .
  $\  
  \end{center}
 Then $A\notin\Qzero$.
\end{proposition}
\begin{proof}
  Let $q=(q_B,t,\dots,t)$ with $t>0$ large. Then $z_B=N(w-q_B)\ge0$ satisfies $Bz_B+q_B=w\ge0$, and $z=(z_B,0)$ is feasible for $(q,A)$ if $t$ is large. Suppose $(w',z)$ solves $(q,A)$. Then $z\ne0$ and the only index with $q_k\le0$ is $i$, so Lemma~\ref{lem:support} gives $z_i>0$. In the frame $A^\alpha$ we have $q^\alpha_\alpha=-Nq_B$, whose only nonpositive entry is the $i$th, and $q^\alpha_k=t-A_{k\alpha}Nq_B>0$ for $k\notin\alpha$ ($t$ large). If $u=(w'_\alpha,z_{\bar\alpha})\ne0$, Lemma~\ref{lem:support} gives $w'_i>0$, contradicting $z_iw'_i=0$. If $u=0$, then $z_\alpha=q^\alpha_\alpha$, so $z_i=-(Nq_B)_i\le0$, contradicting $z_i>0$.
\end{proof}

Theorems~\ref{thm:order2} and~\ref{thm:order3} are instances of Proposition~\ref{prop:criterion} ($i$ is the index of a row of $N$ that is nonnegative and whose column has only negative products with the rest, see the choices of $q_B$ above).

\begin{theorem}[Minimal negative minors of order four, two block shapes]
  \label{thm:order4}
  Let $A\in\Ezerof$ have a negative principal minor of order $4$ on $\alpha$ with all principal minors of order $\le3$ nonnegative, and let $N=(A_{\alpha\alpha})^{-1}\in\Ezero\cap\Nzero$ (Proposition~\ref{prop:minimal}). By \cite[Theorem~3.4]{MurthyParthasarathy1995}, after a principal rearrangement $N=\begin{pmatrix}U&X\\Y&V\end{pmatrix}$ with $U,V$ strictly upper triangular and nonpositive, $X,Y\ge0$, of block sizes $(p,4-p)$, $p\in\{1,2,3\}$. If $p\in\{1,3\}$, then $A\notin\Qzero$.
\end{theorem}
\begin{proof}
  \emph{Case $p=1$.} Write
  $N=\begin{pmatrix}0&x_2&x_3&x_4\\y_2&0&a&b\\y_3&0&0&c\\y_4&0&0&0\end{pmatrix}$ with $x_l,y_l\ge0$, $a,b,c\le0$. Then $\det N=-x_2y_4ac\ne0$, so $x_2,y_4>0$ and $a,c<0$. Take $i=1$, $q_B=(-1,1,1,1)^t$, $w=(0,w_2,0,0)^t$ with $w_2=(x_2+x_3+x_4)/x_2$. Then $(Nq_B)_1=x_2+x_3+x_4\ge0$ and $(Nq_B)_2=-y_2+a+b<0$, $(Nq_B)_3=-y_3+c<0$, $(Nq_B)_4=-y_4<0$. Moreover $N(w-q_B)$ has first entry $x_2(w_2-1)-x_3-x_4=0$ and, for $j\ge2$, $y_j+(\text{nonnegative terms }-V_{jl})\ge0$, because $w_l=0$ for $l\ge3$. Proposition~\ref{prop:criterion} applies.

  \emph{Case $p=3$.} Write $N=\begin{pmatrix}0&a&b&x_1\\0&0&c&x_2\\0&0&0&x_3\\y_1&y_2&y_3&0\end{pmatrix}$ with $\det N=-y_1acx_3\ne0$, so $y_1,x_3>0$, $a,c<0$. Take $i=4$, $q_B=(1,1,1,-1)^t$, $w=(w_1,0,0,0)^t$ with $w_1=(y_1+y_2+y_3)/y_1$. Then $(Nq_B)_4=y_1+y_2+y_3\ge0$, $(Nq_B)_3=-x_3<0$, $(Nq_B)_2=c-x_2<0$, $(Nq_B)_1=a+b-x_1<0$, and $N(w-q_B)$ has fourth entry $0$ and entries $x_j+\sum_{l>j}(-U_{jl})\ge0$ for $j\le3$. Proposition~\ref{prop:criterion} applies.
\end{proof}

\begin{corollary}
  \label{cor:n5b}
  If $A\in\Ezerof\cap\Qzero$ has a negative principal minor and $m$ is the smallest order of a negative principal minor, then $m\ge4$, and if $m=4$ the matrix $N=(A_{\alpha\alpha})^{-1}$ has block shape $(2,2)$. In particular, a counterexample of order $5$ would have a negative $4\times4$ principal minor with block shape $(2,2)$, all principal minors of order $\le3$ nonnegative, and (by minimality of the order) all of its $4\times4$ principal submatrices that lie in $\Qzero$ in $\Pzero$.
\end{corollary}

\begin{remark}[The remaining case]
  For block shape $(2,2)$ the two-frame argument of Proposition~\ref{prop:criterion} fails in general: it fails, for instance, for weighted $4$-cycles with all positive weights or with alternating signs, where all proper principal minors of $B$ vanish. In a numerical experiment, all $1334$ matrices in $\Ezerof$ of order $5$ with entries in $\{-1,0,1\}$ and with the cyclic permutation matrix as leading $4\times4$ block failed a sampled $\Qzero$ test, which is consistent with Conjecture~\ref{conj:stone}; but this is not a proof. Handling shape $(2,2)$ requires pivot frames that involve indices outside $\alpha$ and the extension results of \cite[Section~2]{MurthyParthasarathy1995}.
\end{remark}

Conjecture~\ref{conj:stone} therefore remains open for matrices whose smallest negative principal minor has order $m\ge4$.

\section{Discussion}
\label{sec:conclusion}

We showed that the sign of a principal minor of an $\Ezerof$-matrix is constant on components of the graph $\G(A)$ and positive on the component of the empty set (Theorem~\ref{thm:components}). This yields Conjecture~\ref{conj:stone} for matrices with no vanishing principal minor and for matrices satisfying the descending condition (Corollary~\ref{cor:main}), shows that any negative minor lives in a component separated from $\emptyset$, and, via Propositions~\ref{prop:gap} and~\ref{prop:minimal}, describes the local structure at a minimal negative minor.

The examples show that three natural strategies are blocked. The condition $\det A>0$ cannot force $\Pzero$ within $\Ezerof$ (Example~\ref{ex:JJ}); perturbation $A\mapsto A+\varepsilon I$ leaves $\Ezerof$ (Proposition~\ref{prop:perturb}); and no determinant sign criterion detects $\Qzero$ (Example~\ref{ex:M}). Therefore any proof of Conjecture~\ref{conj:stone} must exploit $\Qzero$ for $A$ itself and must deal with matrices having vanishing principal minors and negative components in $\G(A)$. In view of Remark~\ref{rmk:MP}, a counterexample would be a $\Qzero$-matrix having a non-$\Qzero$ proper principal submatrix of the form in Proposition~\ref{prop:minimal}. We pose the following question: \emph{can the extension results for $\Qzero$-matrices and their principal submatrices of order $n-1$ \cite[Section~2]{MurthyParthasarathy1995} be combined with the component structure of $\G(A)$ to exclude such a submatrix?}

\backmatter

\section*{Declarations}
\begin{itemize}
\item \textbf{Funding} The author is supported by an institute-level fellowship from the Indian Statistical Institute, Delhi Centre, India.
\item \textbf{Competing interests} The author declares no competing interests.
\item \textbf{Data availability} No datasets were generated or analysed during this study. All examples and computations are contained in the manuscript.
\item \textbf{Code availability} Not applicable.
\end{itemize}
\bibliography{sn-bibliography}

@phdthesis{Stone1981,
  author    = {Stone, R. E.},
  title     = {Geometric aspects of the linear complementarity problem},
  school    = {Department of Operations Research, Stanford University},
  address   = {Stanford, CA},
  year      = {1981}
}

@article{SamelsonThrallWesler1958,
  author    = {Samelson, H. and Thrall, R. M. and Wesler, O.},
  title     = {A partition theorem for {E}uclidean $n$-space},
  journal   = {Proc. Amer. Math. Soc.},
  volume    = {9},
  pages     = {805--807},
  year      = {1958}
}

@article{CottleStone1983,
  author    = {Cottle, R. W. and Stone, R. E.},
  title     = {On the uniqueness of solutions to linear complementarity problems},
  journal   = {Math. Program.},
  volume    = {27},
  pages     = {191--213},
  year      = {1983}
}

@article{MurthyParthasarathy1995,
  author    = {Murthy, G. S. R. and Parthasarathy, T.},
  title     = {Some properties of fully semimonotone, {$Q_0$}-matrices},
  journal   = {SIAM J. Matrix Anal. Appl.},
  volume    = {16},
  pages     = {1268--1286},
  year      = {1995}
}

@article{MND2001,
  author    = {Mohan, S. R. and Neogy, S. K. and Das, A. K.},
  title     = {On the classes of fully copositive and fully semimonotone matrices},
  journal   = {Linear Algebra Appl.},
  volume    = {323},
  pages     = {87--97},
  year      = {2001}
}

@article{NeogyDas2005,
  author    = {Neogy, S. K. and Das, A. K.},
  title     = {Principal pivot transforms of some classes of matrices},
  journal   = {Linear Algebra Appl.},
  volume    = {400},
  pages     = {243--252},
  year      = {2005}
}

@article{Tucker1963,
  author    = {Tucker, A. W.},
  title     = {Principal pivotal transforms of square matrices},
  journal   = {SIAM Rev.},
  volume    = {5},
  pages     = {305},
  year      = {1963}
}

@article{Tsatsomeros2000,
  author    = {Tsatsomeros, M. J.},
  title     = {Principal pivot transforms: properties and applications},
  journal   = {Linear Algebra Appl.},
  volume    = {307},
  pages     = {151--165},
  year      = {2000}
}
\end{document}